\documentclass{article}

\usepackage[english]{babel}
\usepackage{graphicx}
\usepackage{amsmath}
\usepackage{amsfonts}
\usepackage{url}
\usepackage[a4paper,margin=3cm]{geometry}

\usepackage{amsthm}

\newtheorem{theorem}{Theorem}[section]
\newtheorem{lemma}[theorem]{Lemma}
\newtheorem{corollary}[theorem]{Corollary}
\theoremstyle{definition}
\newtheorem{definition}[theorem]{Definition}
\newtheorem{remark}[theorem]{Remark}
\newtheorem{problem}{Problem}

\newcommand{\Z}{\mathbb{Z}}
\newcommand{\ff}{\mathfrak{f}}
\newcommand{\Q}{\mathbb{Q}}
\newcommand{\R}{\mathbb{R}}
\newcommand{\F}{\mathbb{F}}
\renewcommand{\L}{\mathcal{L}}

\newcommand{\set}[1]{\left\{{#1}\right\}}
\newcommand{\st}{\mid}
\newcommand{\gen}[1]{\left\langle{#1}\right\rangle}
\newcommand{\quo}[2]{{#1}\big/\raisebox{-0.5ex}{$#2$}}
\newcommand{\freemod}[2]{{#1}^{({#2})}}
\newcommand{\lra}{\mathop{\longrightarrow}\limits}
\newcommand{\ra}{\rightarrow}
\newcommand{\id}{\mathrm{id}}
\renewcommand{\ker}[1]{\mathrm{ker}\left({#1}\right)}
\newcommand{\cok}[1]{\mathrm{cok}\left({#1}\right)}
\newcommand{\im}[1]{\mathrm{im}\left({#1}\right)}
\newcommand{\supp}[1]{\mathrm{supp}{\left({#1}\right)}}
\newcommand{\Dep}[1]{\mathrm{Dep}\left({#1}\right)}
\newcommand{\Hom}[2][]{\ensuremath{\mathrm{Hom}_{#1}\left({#2}\right)}}
\newcommand{\rk}[2][]{\ensuremath{\mathrm{rk}_{#1}\left({#2}\right)}}
\renewcommand{\dim}[2][]{\ensuremath{\mathrm{dim}_{#1}\left({#2}\right)}}
\renewcommand{\gcd}[1]{\ensuremath{\mathrm{gcd}\left({#1}\right)}}
\newcommand{\norm}[1]{\ensuremath{\left|{#1}\right|}}

\begin{document}

\title{An Algebraic Framework for Discrete Tomography:\\Revealing the
Structure of Dependencies}
\author{\small Arjen Stolk$^1$, K. Joost Batenburg$^2$ \\ \scriptsize
$^1$Mathematical Institute, University of Leiden, The Netherlands\\
\scriptsize $^2$Vision Lab, University of Antwerp, Belgium \\ }

\maketitle

\begin{abstract}

Discrete tomography is concerned with the reconstruction of images that
are defined on a discrete set of lattice points from their projections
in several directions. The range of values that can be assigned to each
lattice point is typically a small discrete set. In this paper we
present a framework for studying these problems
from an algebraic perspective, based on Ring Theory and
Commutative Algebra. A principal advantage of this abstract setting is
that a vast body of existing theory becomes accessible for solving
Discrete Tomography problems. We provide proofs of several new results
on the structure of dependencies between projections, including a
discrete analogon of the well-known Helgason-Ludwig consistency
conditions from continuous tomography.

\end{abstract}

\section{Introduction}\label{introduction_section}%

Discrete tomography (DT) is concerned with the reconstruction of
discrete images from their projections.
According to \cite{Herman99,Herman07}, the field of discrete tomography
deals with the reconstruction of images from a small number
of projections, where the set of pixel values is known to
have only a few discrete values. On the other hand, when the
field of discrete tomography was founded by Larry Shepp in
1994, the main focus was on the reconstruction of (usually
binary) images for which the domain is a discrete set, which
seems to be more natural as a characteristic property of discrete
tomography. The number of pixel values may be as
small as two, but reconstruction problems for more values
are also considered. In this paper, we follow the latter
definition of discrete tomography.

Most of the literature on discrete tomography focuses on
the reconstruction of lattice images, that are defined on a
discrete set of points, typically a subset of $\mathbb{Z}^2$.  An image
is formed by assigning a value to each lattice point. The range of these
values is usually restricted to a small, discrete set. The case of
\emph{binary images}, where each point is assigned a value from the set
$\{0,1\}$ is most common in the DT literature. \emph{Projections} of an
image are obtained by summation of the point values along sets of
parallel discrete lines. For an individual line, such a sum is often
referred to as the \emph{line sum}.

Discrete tomography problems have been studied in various fields
of Mathematics, including Combinatorics, Discrete Mathematics and
Combinatorial Optimization. An overview of known results is given
in \cite{Gardner06}, at the end of Section 2.
Already in the 1950s, both Ryser \cite{Ryser57} and Gale \cite{Gale57}
considered the combinatorial problem of reconstructing a binary matrix
from its row and column sums. They provided existence and uniqueness
conditions, as well as concrete reconstruction algorithms. DT emerged as
a field of research in the 1990s, motivated by applications in atomic
resolution electron microscopy \cite{Schwander93, Kisielowski95,
Jinschek07}. Since that time, many fundamental results on the existence,
uniqueness and stability of solutions have been obtained, as well as a
variety of proposed reconstruction algorithms.

Besides purely combinatorial properties, integer numbers play an
important role throughout DT, due to their close connection with the
concepts of reconstruction lattice, lattice line and line sums. A link
with the field of Algebraic Number Theory was established in
\cite{Gardner97}, where Gardner and Gritzmann used Galois theory and
$p$-adic valuations to prove that convex lattice sets are uniquely
determined by their projections in certain finite sets of directions.
Hajdu and Tijdeman described in \cite{Hajdu01} how a powerful extension
of the binary tomography problem is obtained by considering images for
which each point is assigned a value in $\mathbb{Z}$. The fact that both
the image values and the line sums are in $\mathbb{Z}$ allows for the
application of Ring Theory, and in particular the Chinese Remainder
Theorem, for characterizing the set of switching components: images for
which the projections in all given lattice directions are 0. Their
theory for the extended problem leads to new insights in the binary
reconstruction problem as well, as any binary solution must also be a
solution of the extended problem, and the binary solutions can be
characterized as the solutions of the extended problem that have minimal
Euclidean norm.

More recently, techniques from Algebra and Algebraic Number Theory were
used to obtain Discrete Tomography results on stability \cite{Alpers06},
a link between DT and the Prouhet-–Tarry-–Escott problem from Number
Theory \cite{Alpers07}, and the reconstruction of quasicrystals
\cite{Baake06,Gritzmann08}.

In this paper we present a comprehensive framework for the treatment
of DT problems from an
algebraic perspective, based on general Ring Theory and Commutative
Algebra. Modern algebra is a mature mathematical field that provides a
framework in which a wide range of problems can be described, analyzed
and solved.  An important advantage of this abstract setting is that a
vast body of existing theory becomes accessible for solving discrete
tomography problems. Based on our algebraic framework, we provide proofs
of several new results on the structure of dependencies between the
projections, including a discrete analogon of the well-known
Helgason-Ludwig consistency conditions from continuous tomography.

A principal aim of this paper is to create a bridge between the fields
of Combinatorics and classical Number Theory on one side, and the
proposed abstract algebraic model on the other side. To this end, the
definitions and results we describe within our algebraic model will be
followed by concrete examples, illustrating their correspondences with
existing results and concepts.

This paper is organized as follows. In Section
\ref{classical_notation_section} the basic DT problems are introduced in
a combinatorial setting. In Section \ref{global_example} we recall an
example from the literature. Section \ref{algebraic_notation_section}
introduces the same concepts, but this time in our proposed algebraic
framework. We also derive some basic properties linking combinatorial
notions to notions within the framework. Sections
\ref{global_case_section} and \ref{global_ring_section} set up the
algebraic theory, for images defined on $\mathbb{Z}^2$ (the
\emph{global} case).  In Section \ref{example_section} we revisit the
example from Section \ref{global_example} from an algebraic perspective.

In the next sections, the attention is shifted towards images that are
defined on a \emph{subset} of $\mathbb{Z}^2$. Section
\ref{comparison_sequence_section} introduces a relative setup, where a
DT problem on a particular domain is related to a problem on a subset of
that domain. In Sections \ref{finite_convex_section} and
\ref{local_global_dependencies_section}, we apply this relation to
completely describe the structure of line sums for finite convex sets.
The Appendix collects some algebraic results used in the paper.

The authors would like to expres their gratitude towards prof. H.W.
Lenstra for the interesting conversations that led to the development of
our algebraic framework. In particular, prof. Lenstra came up with
Theorem \ref{weak chinese remainder theorem} in an effort to understand
the global dependencies.

\section{Classical definitions and problems}%
\label{classical_notation_section}%

In this section we provide an overview of several important problems in
discrete tomography, within their original combinatorial context.  For
the most part, we follow the basic terminology from \cite{Herman99}.

Let $K\subset\mathbb{Z}$. We will call the elements of $K$
\emph{colours}.  In discrete tomography, we often have $K = \{0,1\}$.
Note that $K$ does not have to be finite.
A nonzero vector $v = (a,b) \in \mathbb{Z}^2$ such that $a \ge 0$ is
called a \emph{lattice direction}. If $a$ and $b$ are coprime, we call
$v$ a \emph{primitive lattice direction}.  The set of all lattice
directions is denoted by $\mathcal{V}$.  For any $t \in \mathbb{Z}^2$,
the set $\ell_{v,t} = \{{\lambda}v + t\,\,| \lambda \in \mathbb{Z}\}$ is
called a \emph{lattice line parallel to} $v$. The set of all lattice
lines parallel to $v$ is denoted by $\mathcal{L}_{v}$.  A function
$f:\mathbb{Z}^2\to K$ with finite support is called a \emph{table}.
The set of all tables is denoted by $\mathcal{F}$. We prefer using the
word table over the more common \emph{image}, as the latter is also used
to denote the image of a map.

\begin{definition}
Let $f \in \mathcal{F}$ and $v \in \mathcal{V}$.
The function $P_{v}(f): \mathcal{L}_v \to \mathbb{Z}$ defined by
\[P_{v}(f)(\ell) = \sum_{x \in \ell}f(x)\]
is called the projection of $f$ in the direction $v$.
\end{definition}
The values $P_{v}(f)(\ell)$ are usually called \emph{line sums}.  For $v
\in \mathcal{V}$, we denote the set of all functions $\mathcal{L}_v \to
\mathbb{Z}$ by $L_v$ (the \emph{potential line sums for direction $v$}).

For a finite ordered set $D = \{v_1,\ldots,v_k\} \subset \mathcal{V}$ of
distinct primitive lattice directions, we define the \emph{projection}
of $f$ along $D$ by
\[P_{D}(f) = P_{v_1}(f) \oplus \ldots \oplus P_{v_k}(f),\]
where $\oplus$ denotes the direct sum. The map $P_{D}$ is called the
\emph{projection map}. Put $L_D = L_{v_1} \oplus \ldots \oplus L_{v_k}$,
the set of \emph{potential line sums for directions $D$}.

Most problems in discrete tomography deal with the reconstruction of
a table $f$ from its projections in a given set of lattice directions. It
is common that a set $A \subset \mathbb{Z}^2$ is given, such that
the support of $f$ must be contained in $A$. We call the set $A$ the
\emph{reconstruction lattice}. Put
$\mathcal{A} = \{f \in \mathcal{F}: x \notin A\implies f(x) = 0\}$.

Similar to Chapter 1 of \cite{Herman99}, we introduce three basic
problems  of DT: Consistency, Reconstruction and Uniqueness:

\vspace{3mm}
\begin{problem}[Consistency]
Let $K$ and $A$ be given.
Let $D = \{v_1,\ldots,v_k\} \subset \mathcal{V}$ be a finite set of
distinct primitive lattice directions and $p \in L_D$ be a given map of
potential line sums.  Does there exist a table $f \in \mathcal{A}$ such
that $P_{D}(f) = p$?
\label{consistency_problem}
\end{problem}

\vspace{3mm}
\begin{problem}[Reconstruction]
Let $K$ and $A$ be given.
Let $D = \{v_1,\ldots,v_k\} \subset \mathcal{V}$ be a finite set of
distinct primitive lattice directions and $p \in L_D$ be a given map of
potential line sums.  Construct a table $f \in \mathcal{A}$ such that
$P_{D}(f) = p$, or decide that no such table exists.
\label{reconstruction_problem}
\end{problem}

\vspace{3mm}
\begin{problem}[Uniqueness]
Given a solution $f$ of Problem \ref{reconstruction_problem}, is there
another solution $g \neq f$ of Problem \ref{reconstruction_problem}?
\label{uniqueness_problem}
\end{problem}

In the most common reconstruction problem in the DT literature, $A$ is a
finite rectangular set of points and $K = \{0,1\}$.  In that case, a
table $f$ is usually considered as a rectangular binary matrix.  For the
case $D = \{(1,0), (0,1)\}$, the three basic problems were solved by
Ryser in the 1950s. It was proved by Gardner et al.  that the
reconstruction problem for more than two lattice directions is NP-hard
\cite{Gardner99}.  Several variants of the reconstruction problem that
make additional assumptions about the table $f$, such as convexity or
periodicity, can be solved effectively if more projections are given
\cite{Barcucci96,Brunetti01}.

Tijdeman and Hajdu considered the case that $A$ is a rectangular set and
$K = \mathbb{Z}$. They show that the resulting problems are strongly
connected to the binary case: if the reconstruction problem for $K =
\mathbb{Z}$ has a binary solution, the set of binary solutions is
exactly the set of tables over $\mathbb{Z}$ for which the Euclidean norm
is minimal.  In \cite{Hajdu01}, they characterized the set of
\emph{switching components}, tables for which the projection is $0$ in
all given lattice directions. In particular, this provides a (partial)
solution for the uniqueness problem, which also has consequences for the
case $K = \{0,1\}$.

\subsection{Dependencies}%
\label{dependency_subsection}%

The theory of Hajdu and Tijdeman also provides insight in the
\emph{dependencies} between the projections of a table, defined below.

If the reconstruction lattice $A$ is finite, the set of lines along
directions in $D$ intersecting with $A$ is also finite. Denote the
number of such lines by $n(A, D)$.
A map $p \in L_D$ of potential line sums can now be represented by an
$n(A,D)$-dimensional vector over $\mathbb{Z}$, where we only consider
the line sums for lines that intersect with $A$. In the remainder of
this section, we use this representation for the projection of a table.

\begin{definition}[Dependency]\label{definition classical dependency}%
Let $A \subset \mathbb{Z}^2$ be a finite reconstruction lattice.
Let $D \subset \mathcal{V}$ be a finite set of distinct primitive
lattice directions.
A \emph{dependency} is a vector $c \in \mathbb{Z}^{n(A,D)}$ such that
for all $f \in \mathcal{F}: P_{D}(f) \cdot c = 0$, where $\cdot$
denotes the vector inner product.
\end{definition}

The vector $c$ is called the \emph{coefficient vector} of the dependency.
Intuitively, dependencies are relations that must always hold between
the set of projections of an object. The simplest such relation corresponds
to the fact that for all lattice directions $v_1,v_2\in\mathcal{V}$:
\[
\sum_{\ell\in\mathcal{L}_{v_1}}P_{v_1}(f)(\ell) =
\sum_{\ell\in\mathcal{L}_{v_2}}P_{v_2}(f)(\ell) =
\sum_{x\in\mathcal{A}}f(x)
\]

More complex dependencies can be formed between sets of three or more
projections.  We call a set of dependencies \emph{independent} if the
corresponding coefficient vectors are linearly independent. Note that
the dependencies form a linear subspace of $\mathbb{Z}^{n(A,D)}$.

\subsection{Example}\label{global_example}%

In \cite{Hajdu01}, the dependencies were systematically investigated for
the case  $K = \mathbb{Q}$,
$A = \{(i,j) \in \mathbb{Z}^2: 0 \leq i < m, 0 \leq j < n\}$ and\\
$D = \{(1,0), (0,1), (1,1), (1,-1)\}$.
Put
\[
\begin{array}{rclll}
r_j &=& \sum_{i=0}\limits^{m-1}f(i,j)\quad\quad & 0 \le j \le n-1,\quad
& \mbox{the row sums,}\\[15pt]
c_i &=& \sum_{j=0}\limits^{n-1}f(i,j)& 0 \le i \le m-1,
& \mbox{the column sums,}\\[20pt]
t_h &=& \mathop{\sum\limits_{j=i+h}}\limits_{(i,j)\in A}f(i,j)
& -m + 1 \leq h < n, & \mbox{the diagonal sums,}\\[20pt]
u_h &=& \mathop{\sum\limits_{j=-i+h}}\limits_{(i,j)\in A}f(i,j)
& 0 \leq h < m + n - 1, & \mbox{the anti-diagonal sums.}
\end{array}
\]

Then the following seven dependencies hold for the line sums:
\begin{eqnarray*}
\sum_{j=0}^{n-1}r_j = \sum_{i=0}^{m-1}s_i \,\, &=&
\sum_{h=-m+1}^{n-1}t_h = \sum_{h=0}^{m+n-2}u_h,\\
\mathop{\sum_{h=-m+1}^{n-1}}_{h\mathrm{\,\,is\,\,odd}}t_h \,\,&=&
\mathop{\sum_{h=0}^{m+n-2}}_{h\mathrm{\,\,is\,\,odd}}u_h,\\
-\sum_{j=0}^{n-1}jr_j + \sum_{i=0}^{m-1}is_i \,\,&=&
\sum_{h=-m+1}^{n-1}ht_h,\\
\sum_{j=0}^{n-1}jr_j + \sum_{i=0}^{m-1}is_i \,\,&=&
\sum_{h=0}^{m+n-2}hu_h,\\
2\sum_{j=0}^{n-1}j^2r_j + 2\sum_{i=0}^{m-1}i^2s_i \,\,&=&
\sum_{h=-m+1}^{n-1}h^2t_h + \sum_{h=0}^{m+n-2}h^2u_h.
\end{eqnarray*}

If $A$ is sufficiently large, these dependencies form an independent
set. It was shown in \cite{Hajdu01} that these relations form a basis of
the space of all dependencies over $\mathbb{Q}$. Although Hajdu and
Tijdeman described the complete set of dependencies for this particular
set of directions, they did not provide a characterization of
dependencies for general sets of directions. They derived a formula for
the dimension of the space of dependencies, for any rectangular set $A$
and any set of directions.

Several properties of the given example deserve further attention.  The
coefficients of the vectors describing the dependencies have the
structure of polynomials in $i$, $j$ and $h$. The degree of these
polynomials is at most two (for the last dependency), and this degree
appears to increase along with the number of directions. In particular,
the maximum degree of the polynomials describing the coefficients in
this example is two, for the dependency involving all four directions,
whereas the maximum degree for a dependency involving any subset of
three directions is one, and the maximum degree for the pairwise
dependencies is zero.

For this set $D$, all of the 7 independent dependencies can be defined
for the case $A = \mathbb{Z}^2$, such that for smaller reconstruction
lattices the same relations hold, restricted to the lines intersecting
$A$. In this paper, we will denote such dependencies by the term
\emph{global dependencies}.

For other sets of directions, such as $D = \{(1,1), (1,2)\}$, there can
also be dependencies such as the one shown in Fig. \ref{corner_figure}.
Two corner points of the reconstruction lattice belong to a line in both
directions, leading to trivial dependencies between the corresponding
line sums. Such dependencies depend on the shape of the reconstruction
lattice and cannot be extended to dependencies on $A = \mathbb{Z}^2$. We
refer to such dependencies as \emph{local dependencies}. An analysis of
the dependencies for the case of a rectangular reconstruction lattice
$A$ is given in
\cite{Vandalen07}.

\begin{figure}[ht!]
\centering
\includegraphics[width=0.4\textwidth]{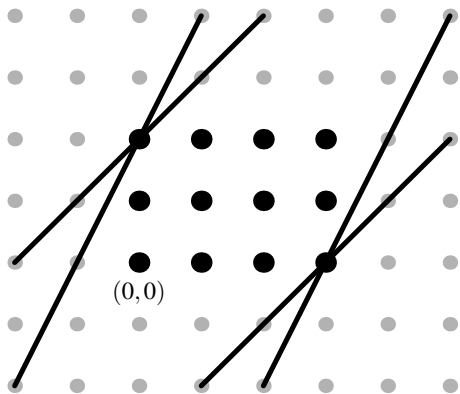}
\caption{At corners of the reconstruction lattice, there can be local
dependencies between line sums in two or more directions.}
\label{corner_figure}
\end{figure}

There is a strong analogy between the concept of dependencies between
line sums in discrete tomography, and so-called \emph{consistency
conditions} in continuous tomography. Ludwig \cite{Ludwig66} and
Helgason \cite{Helgason80} described a set of relations between the
projections of a continuous function defined on $\mathbb{R}^2$.
Moreover, if a set of one-dimensional functions satisfies these
relations, this is also a \emph{sufficient} condition for correspondence
to a projected function.

In the remainder of this paper, we provide a characterization of the
dependencies between projections in discrete tomography, based on our
algebraic framework.  As dependencies indicate relations that must hold
for any set of projections, they provide a necessary condition for the
consistency problem. We prove that for a particular class of discrete
tomography problems, a set of projections satisfies the dependency
relations \emph{if and only if} it corresponds to a table. This leads to
a discrete analogon of the consistency conditions from continuous
tomography.

\section{Algebraic framework}%
\label{algebraic_notation_section}%

In this section we introduce the basic concepts and definitions used in
our algebraic formulation of discrete tomography. For a thorough
introduction to terminology and concepts of Algebra, we refer to
\cite{Lang}. The Appendix of this paper covers some of the properties
used in detail.

Let $A \subset \Z^2$ be non-empty and let $k$ be a commutative ring that
is not the zero ring.
We let
\[T(A, k) = \freemod{k}{A} = \set{ f : A \ra k \st f(x) = 0
\mbox{ for all but finitely many $x \in A$} }\]
be the space of $k$-valued \emph{tables} on $A$. It is a free
$k$-module with a basis indexed by the elements of $A$. We will identify
the elements of $A$ with the elements of this basis.

Let $d \in \Z^2 \setminus \set{0}$ be a direction and $p \in \Z^2$ be
a point. Recall that the \emph{(lattice) line} through $p$ in
the direction $d$ is the set $\set{ p + \lambda d \st \lambda \in \Z}$.
Two points $p$ and $q$ are on the same line in direction $d$ precisely
if they differ by an integer multiple of $d$. The quotient group
$\Z^2/\gen{d}$ therefore parametrises all the lines in the direction $d$.
For $A \subset \Z^2$ write $\L_d(A)$ for the image of
$A$ in $\Z^2/\gen{d}$, i.e. the set of lines in the direction $d$
that intersect $A$.

We call $(a, b) \in \Z^2 \setminus \set 0$ with $\gcd{a,b} = 1$
a \emph{primitive} direction. Whenever $d$ is a primitive direction, the
quotient $\Z^2/\gen{d}$ is isomorphic to $\Z$. This means we can
label the lines in direction $d$ with integers, starting with $0$ for
the line through the origin.

We fix once and for all pairwise independent directions $d_1, \ldots,
d_t \in \Z^2 \setminus \set0$ and write $\L_i(A) = \L_{d_i}(A)$ for the
lines in direction $d_i$ that meet $A$. Let
\[L_i(A,k) = \freemod{k}{\L_i(A)}\]
be the space of \emph{potential line sums} in direction $d_i$ and let
\[L(A, k) = \bigoplus_{i=1}^{t} L_i(A,k)\]
be the full space of potential line sums. These are all free
$k$-modules. A basis for $L_i(A,k)$ is given by $\L_i(A)$ and so a basis
for $L(A,k)$ is given by $\L(A) := \coprod_{i=1}^{t} \L_i(A)$.

\begin{definition}
The \emph{line sum map}
\[\sigma_{A,k} : T(A,k) \lra L(A,k)\]
is defined as the $k$-linear map that sends $x \in A$ to the
vector $(\ell_i)_{i=1}^{t}$, where $\ell_i \in \L_i(A)$ is the line in
direction $d_i$ through $x$.
\end{definition}

The line sum map is the direct sum of the component maps
$ \sigma_{i, A, k}: T(A,k) \ra L_i(A,k)$.

The kernel of the line sum map,
\[\ker{\sigma_{A,k}} = \set{t \in T(A,k) \st \sigma_{A,k}(t) = 0 },\]
identifies the space of \emph{switching components} of the discrete
tomography problem: two tables have the same vector of line sums if and
only if they differ by an element of $\ker{\sigma_{A,k}}$.
We will use the cokernel
\[\cok{\sigma_{A,k}} = \quo{L(A,k)}{\im{\sigma_{A,k}}}\]
to gain insight in the structure of the set of possible line sums
of tables within the full space of potential line sums. In particular,
the cardinality of the cokernel `measures' the difference
between these sets.

\begin{definition}
A $k$-linear \emph{dependency} between line sums is a $k$-linear map
\[r : L(A,k) \lra k\]
such that $r \circ \sigma_{A,k}$ is the zero map.
\end{definition}

Note that such a map gives rise to a map
$\bar{r} : \cok{\sigma_{A,k}} \ra k$
and that conversely any $k$-linear map $\cok{\sigma_{A,k}} \ra k$
gives rise to a dependency. In other words, there is an inclusion
\[ \Hom[k]{\cok{\sigma_{A,k}},k} \subset \Hom[k]{L(A,k), k} \]
whose image is precisely the set of dependencies. We will write
$\Dep{A,k}$ for this subspace.

\begin{remark}\label{weight functions}%
The natural map
\[\begin{array}{rcccc}
W & : &  \Hom[k]{L(A, k), k} & \lra & \left\{ c : \L(A) \ra k \right\} \\
& & \phi & \longmapsto & [\ell \mapsto \phi(\ell)]
\end{array}\]
is a bijection.
\end{remark}

For a $\phi \in \Hom[k]{L(A,k), k}$ we can think of $W(\phi)$ as the
\emph{weight} that $\phi$ assigns to each line in $\L(A)$. For
dependencies this corresponds to the concept of a \emph{coefficient
vector} introduced in Section \ref{dependency_subsection}.
If $r \in \Dep{A,k}$ is a dependency then $W(r)$ corresponds to the
vector $c$ from Definition \ref{definition classical dependency}.

The next lemma gives an example of the link between algebraic properties
of the cokernel and questions concerning the discrete tomography
problem.

\begin{lemma}\label{free cokernel}%
Let $A \subset \Z^2$ and let $k$ be a commutative ring that is not the
zero ring.
Suppose that $\cok{\sigma_{A,k}}$ is a free $k$-module of finite rank $n$.
Then $\Dep{A,k}$ is also a free $k$-module of rank $n$ and for any
$l \in L(A,k)$ we have $l \in \im{\sigma_{A,k}}$ if and only if
$d(l) = 0$ for all $d \in \Dep{A,k}$.
\label{dependency_cok_lemma}
\end{lemma}
\begin{proof}
Let $c_1, \ldots c_n$ be a basis for $\cok{\sigma_{A, k}}$. We can write
any $x \in \cok{\sigma_{A,k}}$ uniquely as $x_1 c_1 + \cdots + x_n c_n$.
The maps $e_i: x \mapsto x_i $
are elements of $\Dep{A,k} = \Hom[k]{\cok{\sigma_{A,k}}, k}$.
We claim that the $e_i$ are a basis for $\Dep{A,k}$. Let $r$ be in
$\Dep{A,k}$. For any $x = \sum x_i c_i$ in $\cok{\sigma_{A,k}}$ we have
\[r(x) = d(\sum x_i c_i) = \sum x_i r(c_i).\]
Put $r_i = r(c_i)$. Then we have $r = \sum r_i e_i$. So the $e_i$
generate $\Dep{A,k}$. Note that the $r_i$ are uniquely determined by
$r$. We conclude that the $e_i$ are a basis of $\Dep{A,k}$.

Note that for all $x \in \cok{\sigma_{A,k}}$, we have $x = \sum e_i(x)
c_i$, so if $d(x) = 0$ for all $d \in \Dep{A,k}$, then $x = 0$. When we
apply this to $x = \bar{l}$ for some $l \in L(A,k)$, we see that
$r(l) = 0$ for all $r \in \Dep{A,k}$ if and only if $\bar{l} = 0$, i.e.
$l \in \im{\sigma_{A,k}}$.
\end{proof}

The lemma that we have just proved can be interpreted as follows.
Whenever we find for some $A$ that $\cok{\sigma_{A,k}}$ is a free
$k$-module of finite rank, we have the following:
\textit{A vector of potential line sums comes from a table precisely if
it satisfies all dependencies.}
As the space of dependencies is also free and of finite rank, it in fact
suffices to check finitely many dependencies.

\section{The global case}%
\label{global_case_section}%

In this section we consider the case $A = \Z^2$. We will show that in
this case, the objects defined in the previous section have the
structure of rings and modules, and their homomorphisms. This allows us
to completely describe the kernel and cokernel of the line sum map.

The following three $k$-modules are isomorphic in a natural way:
\[T(\Z^2, k) \qquad \cong \qquad k[\Z^2]
\qquad \cong \qquad k[u, u^{-1}, v, v^{-1}].\]
For some basic properties of group rings such as $k[\Z^2]$, see the
appendix of this article. The isomorphisms are
\[\begin{array}{ccc}
T(\Z^2,k) & \lra & k[\Z^2] \\
\relax[c : \Z^2 \ra k] & \longmapsto & \sum_{x \in \Z^2} c(x) x
\end{array}\]
and
\[\begin{array}{ccc}
k[\Z^2] & \lra & k[u,u^{-1},v,v^{-1}] \\
\sum_{x \in \Z^2} \lambda_x x & \longmapsto &
\sum_{(a,b) \in \Z^2} \lambda_{(a,b)} u^av^b.
\end{array}\]
Note that $k[\Z^2]$ and $k[u,u^{-1},v,v^{-1}]$ are both $k$-algebras and
that the second isomorphism is an isomorphism of $k$-algebras. We also
view $T(\Z^2, k)$ as a $k$-algebra via these isomorphisms.

In the same way there is a natural isomorphism of $k$-modules
\[ L_i(\Z^2, k) \cong k\left[\quo{\Z^2}{\gen{d_i}}\right] \]
which puts a ring structure on the spaces of potential line sums. By
Lemma \ref{groupring quotients} we have an isomorphism
$k[\Z^2/\gen{d_i}] \cong k[\Z^2]/(d_i - 1)$.
Viewed in this way, the line sum map
$\sigma_{i, \Z^2, k} : T(\Z^2,k) \ra L_i(\Z^2,k)$ is
the quotient map
\[k[\Z^2] \ra \quo{k[\Z^2]}{(d_i - 1)}.\]

Taking sums, we find a $k$-algebra structure on $L(\Z^2, k)$ such that
the line sum map $\sigma_{\Z^2, k} : T(\Z^2, k) \ra L(\Z^2, k)$ is a
$k$-algebra map which is the direct sum of quotient maps. We will now
study the structure of these quotient maps from an algebraic perspective
using the ideas outlined in the last part of the appendix.

\begin{lemma}\label{independent directions}%
Let $d, e \in \Z^2$ be independent directions. Then $d - 1$ is
weakly coprime (see \ref{weakly coprime} in the appendix)
to $e - 1$ in $k[\Z^2]$.
\end{lemma}
\begin{proof}
By Lemma \ref{groupring quotients} we can see
$k[\Z^2]/(d - 1)$ as the group ring $k[\Z^2/\gen{d}]$.
Suppose we have
\[f = \sum_{x \in \Z^2/\gen{d}} f_x x \in
k\left[\quo{\Z^2}{\gen{d}}\right]\]
such that $(e-1) f = 0$. When we expand
\[0 = (e - 1) f = \sum_{x \in \Z^2/\gen{d}} (f_{x-e} - f_x) x,\]
we see that $f_{x + k e} = f_x$ for all $x \in \Z^2/\gen{d}$ and
$k \in \Z$.  As $d$ and $e$ are independent, all $x + k e$ are different
in $\Z^2/\gen{d}$. We conclude that we must have $f_x = 0$ for all
$x \in \Z^2/\gen{d}$, as only finitely many coefficients of $f$ are
non-zero.
\end{proof}

\begin{theorem}\label{global_kernel_and_cokernel}%
The kernel of $\sigma_{\Z^2,k}$ is given by
\[ \ker{\sigma_{\Z^2, k}} = (d_1 - 1)\cdots(d_t - 1) k[\Z^2].\]
The cokernel $\cok{\sigma_{\Z^2,k}}$ is a free $k$-module of rank
\[\sum_{1 \le i < j \le t} |\!\det(d_i, d_j)|.\]
\end{theorem}
\begin{proof}
By Lemma \ref{independent directions}, $d_i - 1$ is weakly coprime to
$d_j - 1$ in $k[\Z^2]$ whenever $i \neq j$. So we can apply Theorem
\ref{weak chinese remainder theorem} to the map
\[
\sigma_{\Z^2, k} : k[\Z^2] \lra \bigoplus_{i=0}^{t} \quo{\Z^2}{d_i - 1}.
\]
This immediately gives us the formula for the kernel given in the
theorem. For the cokernel, we note that by Lemma \ref{groupring
quotients} we have
\[
\quo{k[\Z^2]}{(d_i - 1, d_j - 1)} =
k\left[\quo{\Z^2}{\gen{d_i,d_j}}\right],
\]
which is a free $k$-module of rank $|\det(d_i,d_j)|$.
In particular, all the successive quotients of the filtration on the
cokernel are free $k$-modules. Therefore all the quotients are split
(see, e.g., \cite[Ch. III.3, Prop. 3.2]{Lang}) and we conclude that
\[\cok{\sigma_{\Z^2, k}} \cong \bigoplus_{1 \le i < j \le t}
k\left[\quo{\Z^2}{\gen{d_i,d_j}}\right].\]
\end{proof}

This result leads to a (partial) discrete analogon of
the Helgason-Ludwig consistency conditions from continuous
tomography, providing a necessary and sufficient condition
for consistency of a vector of potential line sums:

\begin{corollary}\label{global_consistency_problem}
A vector of potential line sums in $L(\Z^2, k)$ comes from a table in
$T(\Z^2, k)$ if and only if it satisfies all dependencies. Moreover, we
only have to check this for a set of
$\sum_{1 \le i < j \le t} |\!\det(d_i, d_j)|$
independent dependencies.
\end{corollary}
\begin{proof} Theorem \ref{global_kernel_and_cokernel} shows that we can
apply Lemma \ref{free cokernel} to the global cokernel.
\end{proof}

Looking at example \ref{global_example} we compute
$\sum_{1 \le i < j \le 4} |\!\det(d_i, d_j)| = 7$. This tells us that
the list of 7 independent dependencies we had is complete, in the sense
that at least when $k$ is a field, they will form a basis of
$\Dep{\Z^2, k}$.

For a full discrete analogon of the continuous consistency conditions,
one should also provide a charaterization of the structure of the
individual dependencies.  The next section provides additional insight
into the coefficient structure of the dependencies.

\section{The global line sum map as an extension of rings}%
\label{global_ring_section}%

We now focus our attention more on the ring theoretic aspect of the
line sum map. We can view $L(\Z^2, k)$ as an extension of its subring
$\im{\sigma_{\Z^2,k}}$. Both these rings have relative dimension $1$
over $k$. This is a situation that has been extensively studied because
of its relation to Algebraic Number Theory. An important object in this
context is the \emph{conductor} of the extension, the largest ideal of
$L(\Z^2, k)$ that is also an ideal of $\im{\sigma_{\Z^2,k}}$.

\begin{lemma}
Put $D_i = \prod_{j \neq i} (d_j - 1)$. The conductor of $L(\Z^2, k)$
over $\im{\sigma_{\Z^2,k}}$ is given by
\[\ff_k =
\overline{D_1} \quo{k[\Z^2]}{d_1 - 1}
\oplus\cdots\oplus
\overline{D_t} \quo{k[\Z^2]}{d_t - 1}.\]
\end{lemma}
\begin{proof}
Note that $D_i$ reduces to $0$ in $k[\Z^2] /(d_j - 1)$ for all $j \neq i$.
We conclude that the ideal $(D_1, \ldots, D_t)$ of $k[\Z^2]$ is mapped by
$\sigma_{\Z^2,k}$ onto $\ff_k$. In particular, this implies that $\ff_k$
is indeed an $\im{\sigma_{\Z^2, k}}$ ideal.

Conversely, suppose $I \subset \im{\sigma_{\Z^2, k}}$ is an ideal that
is also closed under multiplication by $L(\Z^2, k)$. We want to show
that $I \subset \ff_k$. Let $x = (x_1, \dots, x_t) \in I$. As $I$ is an
$L(\Z^2, k)$ ideal, we must also have
$(0, \ldots, x_i, \ldots, 0) \in I$. As $I \subset \im{\sigma_{\Z^2,
k}}$ there is an $\tilde{x}_i \in k[\Z^2]$ such that
$\sigma_{\Z^2, k}(\tilde{x}_i) = (0, \ldots, x_i, \ldots, 0)$. We have
$x = \sigma_{\Z^2, k}(\tilde{x}_1 + \cdots + \tilde{x}_t)$, so we are
done if we can show that $\tilde{x}_i$ is a multiple of $D_i$ for all
$i$.

To show this, we apply Theorem \ref{global_kernel_and_cokernel} to the
directions $d_j$ with $j \neq i$. Note that $\tilde{x}_i$ maps to $0$
under the line sum map in this case. The theorem tells us that the
kernel of this map is generated by $D_i$, so that indeed $\tilde{x}_i$
is a multiple of $D_i$ for all $i$.
\end{proof}

Note that the quotient module $L(\Z^2,k)/\ff_k$ is a free
$k$-module of dimension $\sum_{i \neq j} |\det(d_i, d_j)|$. This is
twice the dimension of $\cok{\sigma_{\Z^2,k}} = L(Z^2,k) /
\im{\sigma_{\Z^2,k}}$. We see that $\im{\sigma_{\Z^2,k}}$ sits
precisely in the middle between $L(\Z^2,k)$ and $\ff_k$. This is not a
surprise, it happens in this situation whenever the rings are
`sufficiently nice,' e.g. when they are Gorenstein rings.

We have not yet fully explored the implications of this ring theoretic
view for the structure of $\cok{\sigma_{\Z^2,k}}$, but we believe it
warrants further investigation. To illustrate its use, we will derive
the following result on the coefficient functions of
dependencies in $\Dep{\Z^2,k}$.

For the remainder of this section, we assume that all the $d_i$ are
primitive directions. This means that $\Z^2 / \gen{d_i}$ is isomorphic
to $\Z$. For the rest of this section we also fix isomorphisms
$\Z^2 / \gen{d_i} \cong \Z$. What this means is that the lines in each
direction $d_i$ can be numbered in sequence. The choice of isomorphisms
comes down to picking whether we number from left to right or the other
way around.

Recall from Remark \ref{weight functions} that a dependency
$r \in \Dep{Z^2, k}$ can be represented by a function $W(r)$ from
$\L(\Z^2)$ to $k$. From the choices we have just made, $\L(\Z^2)$ is
identified with $t$ copies of $\Z$. This means that to represent a
dependency by a set of $t$ two-sided infinite sequences
\[ W_i(r) : \Z \lra k. \]

\begin{theorem} With the assumptions above, each sequence
$W_i(r)$ satisfies a non-trivial linear recurrence relation that does
not depend on $r$.
\end{theorem}
\begin{proof}
The isomorphism $\Z^2 / \gen{d_i} \cong \Z$ gives rise to an isomorphism
\[\quo{k[\Z^2]}{d_i - 1} \cong k\left[\quo{Z^2}{\gen{d_i}}\right]
\cong k[x, x^{-1}] \]
of $L_i(\Z^2, k)$ with the Laurent polynomial ring $k[x,x^{-1}]$. Write
$\overline{D_i} = \sum_n a_n x^n $ in $k[x, x^{-1}]$.

Let $r \in \Dep{\Z^2, k}$ be a dependency. We consider the map
$r_i: L_i(\Z^2, k) \ra k$ induced by $r$. As $\im{\sigma_{\Z^2, k}}$ is
in the kernel of $r$, we have $\ff_k \subset \ker{r}$. As
$x \in k[x, x^{-1}]$ is a unit, we see that $x^n \overline{D_i}$ must be in
$\ker{r_i}$ for all integers $n$.

Write $c$ for the weight function $W_i(r)$ from $\Z$ to $k$. From the
definitions, we have for all $n \in \Z$ that $r_i(x^n) = c(n)$.
Let $m \in Z$. Then we must have
\[ 0 = r_i(x^m \overline{D_i}) =
r_i(\sum_n a_n x^{m + n}) = \sum_n a_n c(m + n).\]
This is saying precisely what we want, namely that $c$ satisfies a
linear recurrence relation whose coefficients are the $a_n$. Clearly,
these $a_n$ do not depend on $r$, only on $D_i$ and maybe on $k$.
\end{proof}

In fact, one computes that for $i = 1, \ldots t$ we have
\[ \overline{D_i} = \prod_{j \neq i} (x^{\det(d_i, d_j)} - 1).\]
From this, one easily sees that the leading and trailing coefficients of
$\overline{D_i}$ are $\pm 1$. Therefore, no matter what $k$ is, the
recurrence relation can be used to uniquely determine the sequence from
any sufficiently large set of consecutive coefficients. In fact, all the
coefficient functions can be expressed in a closed form
\[ [W_i(r)](s) = f_{s\ \mathrm{mod}\ m}(s)\]
where the $f$ are polynomials. The maximal degrees of these polynomials and
the value of $m$ depend only on the $d_i$ and the characteristic of $k$.

\section{An example}%
\label{example_section}%
We revisit the example from \cite{Hajdu01} that was discussed in Section
\ref{global_example}. It concerns the directions $d_1 = (1,0)$, $d_2 =
(0,1)$, $d_3 = (1,1)$ and $d_4 = (1,-1)$. For simplicity, we take $k =
\Q$, but we will make some comments on how to deal with the case $k =
\Z$.

We identify $T(\Z^2,k)$ with $k[x,x^{-1},y,v^{-1}]$.
Note that for each $i$, we have $\Z^2 / \gen{d} \cong \Z$. We pick
isomorphisms $L_i(\Z^2, k) = k[z,z^{-1}]$ in such a way that the
components of the line sum map are the maps
$k[x,x^{-1},y,y^{-1}] \ra k[z,z^{-1}]$ given by
\[\begin{array}{c|c|c|c}
i & \mbox{map} & x \mapsto & y \mapsto \\
\hline
1 & r & 1 & z \\
2 & c & z & 1 \\
3 & t & z & z^{-1} \\
4 & u & z & z
\end{array}\]
The line sum map is given by
\[
\sigma = (r,c,t,u) : k[x,x^{-1},y,y^{-1}] \lra \left(k[z,z^{-1}]\right)^4
\]

The maps $r, c, t$ and $u$ are related to the line sums described in
Section \ref{global_example} in a straightforward manner. Let $f$ and
the $r_i, c_i, t_i$ and $u_i$ be as in that section. Put
$F = \sum_{i,j} f(i,j)x^iy^j$. Then we have $r(F) = \sum_i r_i z^i$ and
likewise for the other maps.

We compute
\[\begin{array}{lcl}
D_1 = (y - 1)(xy - 1)(xy^{-1} - 1) & \quad &
r(D_1) = -z^{-1}(z-1)^3 \\
D_2 = (x - 1)(xy - 1)(xy^{-1} - 1) &&
c(D_2) = (z - 1)^3 \\
D_3 = (x - 1)(y - 1)(xy^{-1} - 1) &&
t(D_3) = (z - 1)^3(z + 1) \\
D_4 = (x - 1)(y - 1)(xy - 1) &&
u(D_4) = (z-1)^3(z+1)
\end{array}\]

Let $M = M_1 \oplus \cdots \oplus M_4$ be the quotient vector space
\[M = \frac{k[z,z^{-1}]}{r(D_1)} \oplus
\frac{k[z,z^{-1}]}{c(D_2)} \oplus
\frac{k[z,z^{-1}]}{t(D_3)} \oplus
\frac{k[z,z^{-1}]}{u(D_4)} \]
and $\pi = (\pi_1,\ldots,\pi_4)$ be the quotient map
$(k[z,z^{-1}])^4 \ra M$.
As discussed in the previous section, there is a surjective map
$M \ra \cok{\sigma}$. This means we can realize $\Dep{\Z^2, k}$ as a
subspace of $\Hom{M, k}$.

A basis for $\Hom{k[z,z^{-1}]/(z-1)^3, k}$ is given by the maps
\[v_1 : z^i \mapsto 1 \qquad
v_2 : z^i \mapsto i \qquad
v_3 : z^i \mapsto i^2. \]
Let $e : \Z \mapsto \Z$ be the map that sends $n$ to $0$ if
$n$ is odd, and to $1$ if it is even.
A basis for $\Hom{k[z,z^{-1}]/(z-1)^3(z+1), k}$ is given by
\[ w_1 : z^i \mapsto e(i) \qquad
w_2 : z^i \mapsto 1-e(i) \qquad
w_3 : z^i \mapsto i \qquad
w_4 : z^i \mapsto i^2.\]
These maps together give a basis for $\Hom{M,k}$ consisting of $14$
elements:
\begin{itemize}
\item $v_{1,1}$, $v_{1,2}$ and $v_{1,3}$ acting on the first coordinate;
\item $v_{2,1}$, $v_{2,2}$ and $v_{2,3}$ acting on the second coordinate;
\item $w_{1,1}, \ldots, w_{1,4}$ acting on the third coordinate and
\item $w_{2,1}, \ldots, w_{2,4}$ acting on the fourth coordinate.
\end{itemize}
These maps correspond to the sums of line sums that also come up in
Section \ref{global_example}. For example $v_{1,1}$ sends $F$ to
$\sum_i r_i$ and $w_{2,3}$ sends $F$ to $\sum_i i^2 u_i$.

The dependencies form a subvector space of $\Hom{M,k}$ of dimension
$7$. What we still have to do is to determine which linear combinations
of $v_{i,j}$'s and $w_{i,j}$'s correspond to dependencies. One way to do
this is to write down the restrictions coming from the fact that tables
of the form $x^iy^j$ must be sent to $0$ by a dependency. We will see in
Section \ref{finite_convex_section} that we only have to check finitely
many such tables before we have a complete set of restrictions.

Another way to find these restriction is to consider the compositions of
the $v$'s and $w$'s with $\pi \circ \sigma$, i.e., the maps they induce
in $\Hom{k[x,x^{-1},y,y^{-1}], k}$. The dependencies are precisely those
relations that go to $0$ under this composition. The maps we obtain in
this way are
\[\begin{array}{c|ccc|c}
\mbox{map} & x^iy^j \mapsto & \qquad &
\mbox{map} & x^iy^j \mapsto \\
\hline
v_{1,1} & 1 & & v_{2,1} & 1 \\
v_{1,2} & i & & v_{2,1} & j \\
v_{1,3} & i^2 & & v_{2,1} & j^2 \\
w_{1,1} & e(i-j) & & w_{2,1} & e(i+j) \\
w_{1,2} & 1-e(i-j) & & w_{2,2} & 1-e(i+j) \\
w_{1,3} & i-j & & w_{2,3} & i+j \\
w_{1,3} & (i-j)^2 & & w_{2,3} & (i+j)^2
\end{array}\]

From this table, one easily reads off a basis for the dependencies. For
example, we can take
\begin{eqnarray*}
v_{1,1} = v_{2,1} &=& w_{1,1} + w_{1,2} = w_{2,1} + w_{2,2} \\
w_{1,1} &=& w_{2,1} \\
v_{1,2} - v_{2,2} &=& w_{1,3} \\
v_{1,2} + v_{2,2} &=& w_{2,3} \\
2 v_{1,3} + 2 v_{2,3} &=& w_{1,4} + w_{2,4}.
\end{eqnarray*}
These correspond to the dependencies described in Section
\ref{global_example}.

If we want to write down a basis for the dependencies not over $\Q$ but
over $\Z$ or some other ring, we have to be a little more careful. The
maps $v_1, \ldots v_3$ do not form a basis of $\Hom{k[z,z^{-1}]/(z-1)^3,
k}$ if $k = \Z$. The map sending $z^i$ to $\frac12i(i-1)$ is in this
module, but it is equal to $\frac12(v_3 - v_2)$, which is not a
$\Z$-linear combination of the $v$'s.

A basis that works regardless of
the ring $k$ is found as follows. Note that
\[\quo{k[z,z^{-1}]}{(z-1)^3} =
k \cdot 1 \oplus k \cdot z \oplus k \cdot z^2 \]
This choise of a basis also gives a basis for the $k$-dual. This basis
works independently of $k$. The price we pay for this more general
approach is that the formulas that come out aren't as nice, making it
harder to find the dependencies by hand. The linear algebra involved
does not become more difficult.

\section{The comparison sequence}%
\label{comparison_sequence_section}%

Let $A \subset B \subset \Z^2$. Our aim in this section is to compare
the kernels and cokernels of $\sigma_{A,k}$ and $\sigma_{B,k}$.

Put $T(B/A, k) = \freemod{k}{B \setminus A}$ and
$L(B/A, k) = \bigoplus_{i=1}^{t}\freemod{k}{\L_i(B) \setminus \L_i(A)}$.
Looking at the bases for the spaces involved, it is clear that there are
direct sum decompositions
$T(B,k) = T(A,k) \oplus T(B/A,k)$ and $L(B,k) = L(A,k) \oplus L(B/A,k)$.

This means we can represent $\sigma_{B,k}$ as a two-by-two matrix of
$k$-linear maps
\[ \sigma_{B,k} =
\left(\begin{matrix} p & q \\ r & s\end{matrix} \right),\]
where $p : T(A,k) \ra L(A,k)$, $q: T(B/A, k) \ra L(A, k)$,
$r : T(A) \ra L(B/A,k)$, and $s: T(B/A, k) \ra L(B/A,k)$ are the
restrictions and projections of $\sigma_{B,k}$ to the appropriate
subspaces. The usual matrix multiplication rule
\[\left(\begin{matrix} p & q \\ r & s\end{matrix} \right)
\left(\begin{matrix}x\\y\end{matrix}\right) =
\left(\begin{matrix}u\\v\end{matrix}\right)
\]
holds when we have $x \in T(A, k)$, $y \in T(B/A, k)$, $u \in L(A, k)$,
and $v \in L(B/A, k)$ such that $\sigma_{B,k}(x \oplus y) = u \oplus v$.

As $L(B/A, k)$ consists precisely of those lines through $B$ that do not
intersect $A$, we have $r = 0$. Similarly, $p$ is just the map
sending tables on $A$ to their line sums, so $p = \sigma_{A,k}$.
The other two maps, $q$ and $s$ encode interesting information about the
relative situation, so we will give them more descriptive names
\[\sigma_{B/A,k} : T(B/A, k) \lra L(B/A, k)
\quad\mbox{(the \emph{relative line sum map})}\]
and
\[\delta_{B/A,k} : T(B/A, k) \lra L(A,k)
\quad\mbox{(the \emph{interference map})}.\]

\begin{lemma}[The comparison sequence]
There is a long exact sequence
\[
\begin{array}{l}
0 \ra \ker{\sigma_{A,k}}
\ra \ker{\sigma_{B,k}}
\ra \ker{\sigma_{B/A,k}}\\
\quad\ra \cok{\sigma_{A,k}}
\ra \cok{\sigma_{B,k}}
\ra \cok{\sigma_{B/A,k}}
\ra 0.
\end{array}
\]
The map
$ \overline{\delta_{B/A,k}} : \ker{\sigma_{B/A,k}} \ra \cok{\sigma_{A,k}}$
comes from the interference map $\delta_{B/A, k}$ defined above.
\end{lemma}
\begin{proof}
This is an application of the Snake Lemma
(See, for example, \cite[Ch. III.9, Lemma 9.1.]{Lang}).
\end{proof}

The extension $B/A$ is called \emph{non-interfering} if it satisfies
the following (equivalent) conditions:
\begin{enumerate}
\item the map $\overline{\delta_{B/A,k}}$ is the zero map;
\item the map $\ker{\sigma_{B,k}} \ra \ker{\sigma_{B/A,k}}$ is surjective;
\item the map $\cok{\sigma_{A,k}} \ra \cok{\sigma_{B,k}}$ is injective.
\end{enumerate}

\section{Finite, convex $A$}%
\label{finite_convex_section}%

A subset $C \subset \R^2$ is called \emph{convex} if for any
$x, y \in C$ the line segment between $x$ and $y$ is completely
contained in $C$. The \emph{convex hull} of a subset $S \subset \R^2$ is
the smallest convex subset $C$ of $\R^2$ containing $S$. We write $H(S)$
for the convex hull of $S$.
We call $A \subset \Z^2$ convex if $A = H(A) \cap \Z^2$.

We call $C \subset \R^2$ a \emph{convex polygon} if $C = H(S)$ for some
finite $S \subset \R^2$. The set of \emph{corners} of a convex polygon
$C$ is the smallest set $S$ such that $H(S) = C$.

Let $C_1, C_2 \subset \R^2$ be convex polygons. Then
\[C_1 + C_2 = \set{c_1 + c_2 \st c_1 \in C_1, c_2 \in C_2}\]
is also a convex polygon.  Let $s$ be a corner of $C_1 + C_2$. Then $s$
can be written in a unique way as $s_1 + s_2$ with $s_1 \in C_1$ and
$s_2 \in C_2$. Moreover, $s_1$ and $s_2$ are corners of $C_1$ and $C_2$
respectively.

Let $f \in k[\Z^2]$ and write $f = \sum_{x \in \Z_2} f_x x$. Then the
\emph{support} of $f$ is the set
\[\supp{f}=\set{x \in \Z^2 \st f_x \neq 0}.\]
Note that $\supp{f}$ is always a finite set.
The \emph{polygon} of $f$ is
\[P(f) = H(\supp{f}).\]
It is a convex polygon. Let $s$ be a corner of $P(f)$, then we say that
$s$ is a \emph{strong} corner of $P(f)$ if $f_s$ is not a zero divisor.
We say that $f$ \emph{has strong corners} if all corners of $P(f)$ are
strong.

\begin{lemma}
Let $f, g \in k[\Z^2]$ and suppose that $f$ has strong corners. Then
\[P(fg) = P(f) + P(g).\]
If $g$ also has strong corners, $fg$ has strong corners.
\end{lemma}
\begin{proof}
The inclusion $P(fg) \subset P(f) + P(g)$ is obvious. For the other
inclusion, suppose that $s$ is a corner of $P(f) + P(g)$. Then
the coefficient of $fg$ at $s$ is
\[\sum_{a + b = s} f_a g_b = f_{s_f} g_{s_g}, \]
where $s_f$ and $s_g$ are the unique corners of $P(f)$ and $P(g)$
respectively such that $s = s_f + s_g$. We see that this coefficient is
non-zero as $f_{s_f}$ is not a zero divisor, so $s \in P(fg)$. This
shows that $P(f) + P(g) \subset P(fg)$. Moreover, if $g$ also has strong
corners, $g_{s_g}$ is also not a zero divisor and so $f_{s_f} g_{s_g}$
is not a zero divisor.
\end{proof}

\begin{lemma}\label{kernel generator}%
The generator of $\ker{\sigma_{\Z^2, k}}$,
\[D = (d_1 - 1) \cdots (d_t - 1),\]
has strong corners. Moreover, $\Delta = P(D)$ does not depend on $k$.
\end{lemma}
\begin{proof}
The polygon of $d_i - 1$ is a $2$-gon with coefficients $\pm 1$ at the
corners, so $d_i - 1$ has strong corners. The previous lemma then
implies that $D$ has strong corners.

Let $D_\Z = (d_1 - 1) \cdots (d_t - 1) \in \Z[\Z^2]$, then $D$ is the
image of $D_\Z$ under the natural map $\Z[\Z^2] \ra k[\Z^2]$. Note that
the corners of $D_\Z$ will have coefficients $\pm 1$, as this is true
for all the factors $d_1 - 1$. This means that $P(D) = P(D_\Z)$ does not
depend on $k$, as $\pm 1$ never maps to $0$ in $k$.
\end{proof}

\begin{theorem}
\label{local_kernel_and_cokernel}
Let $A \subset \Z^2$ be finite and convex. Then $\ker{\sigma_{A,k}}$ and
$\cok{\sigma_{A,k}}$ are free $k$-modules of finite rank. The ranks of
these modules do not depend on $k$.
\end{theorem}
\begin{proof}
Note that $\sigma_{A,k}$ is the restriction of $\sigma_{\Z^2, k}$ to
$A$, and so we have
\[\ker{\sigma_{A,k}} = \ker{\sigma_{\Z^2,k}} \cap T(A,k). \]
Using this, we compute
\[
\begin{array}{rcl}
\ker{\sigma_{A,k}}
& = & \ker{\sigma_{\Z^2, k}} \cap T(A,k) \\
& = & D k[\Z^2] \cap T(A,k) \\
& = & \set{ f \in D k[\Z^2] \st \supp{f} \subset A } \\
& = & \set{ f \in D k[\Z^2] \st P(f) \subset H(A) } \\
& = & \set{ f D \st f \in k[\Z^2], P(f D) \subset H(A) } \\
& = & \set{ f D \st f \in k[\Z^2], P(f) + \Delta \subset H(A) }.
\end{array}
\]
The latter is clearly a free $k$-module of finite rank with a basis
indexed by the $x \in \Z^2$ such that $x + \Delta \subset H(A)$. By
Lemma \ref{kernel generator}, this basis is independent of $k$. Therefore
the rank of $\ker{\sigma_{A,k}}$ does not depend on $k$.

This proves the result for the kernel. The result for the cokernel now
follows from algebraic generalities. It suffices to show that
$\cok{\sigma_{A,\Z}}$ is a free $\Z$-module of finite rank, as taking
cokernels commutes with taking tensor products (see e.g.
\cite[Ch. XVI.2, Prop. 2.6]{Lang}.)
Since it is clearly finitely generated, we must show that it is
torsion-free \cite[Ch. I.8, Thm. 8.4]{Lang}. We do this by comparing
the ranks over $\F_p$ for $p$ prime to the rank over $\Z$.

From the sequence
\[ 0 \ra \ker{\sigma_{A,\Z}} \ra T(A,\Z)
\ra L(A,\Z) \ra \cok{\sigma_{A,\Z}} \ra 0 \]
we see that
\[ \rk[\Z]{\cok{\sigma_{A,\Z}}} =
\rk[\Z]{\ker{\sigma_{A,\Z}}} - \#A + \sum_{i=1}^t \#\L_i(A).\]
In the same way, we have for any prime $p$
\[\dim[\F_p]{\cok{\sigma_{A,\F_p}}} =
\dim[\F_p]{\ker{\sigma_{A,\F_p}}} - \#A + \sum_{i=1}^t \#\L_i(A).\]
By the result about the kernel, we know that
$\rk[\Z]{\ker{\sigma_{A,\Z}}} = \dim[\F_p]{\ker{\sigma_{A,\F_p}}}$.
Using the formulas above this implies
\[\rk[\Z]{\cok{\sigma_{A,\Z}}} = \dim[\F_p]{\cok{\sigma_{A,\F_p}}}.\]
But if $\cok{\sigma_{A,\Z}}$ has any $p$-torsion, the $\F_p$-dimension
would be strictly bigger. We conclude that $\cok{\sigma_{A,\Z}}$ is
torsion-free.
\end{proof}

Similar to the global case ($A = \mathbb{Z}^2$), this result allows
to state a necessary and sufficient condition for consistency of
a vector of potential line sums in the case of finite convex $A$:

\begin{corollary}\label{local_consistency_problem}
Let $A \subset \Z^2$ be finite and convex. A vector of potential line
sums in $L(A, k)$ comes from a table in $T(A, k)$ if and only if it
satisfies all dependencies.
\end{corollary}
\begin{proof} Theorem \ref{local_kernel_and_cokernel} shows that we can
apply Lemma \ref{free cokernel} to the cokernel of the line sum map.
\end{proof}

\section{Local and global dependencies}%
\label{local_global_dependencies_section}%

Let $A \subset \Z^2$. From the comparison sequence we have a map
$ \cok{\sigma_{A,k}} \ra \cok{\sigma_{\Z^2, k}}.$
This map induces a map on the $k$-duals
\[\Dep{\Z^2, k} \lra \Dep{A, k}.\]
We call the image of this map the \emph{global dependencies} on $A$.
When this map is injective, the dependencies on $\Z^2$ all restrict to
different dependencies on $A$. Our intuition is that this should happen
whenever $A$ is `sufficiently large.'

\begin{lemma}\label{injective global dependencies}
Suppose there is an $x \in \Z^2$ such that $x + \Delta \subset H(A)$.
Then $\cok{\sigma_{A,k}} \ra \cok{\sigma_{\Z^2,k}}$
is surjective and so
\[\Dep{\Z^2, k} \lra \Dep{A, k}\]
is injective.
\end{lemma}

The \emph{geometric line} through $p \in \R^2$ in the direction
$d \in \R^2 \setminus \set{0}$ is the set
\[\set{p + \lambda d \st \lambda \in \R} \cap \Z^2,\]
provided this set contains at least two points.

Let $d = (a,b) \in \Z^2 \setminus \set{0}$ and put $g = \gcd{a,b}$. Then
any geometric line in direction $d$ is the union of $g$ lines. If $l$ is
a geometric line in direction $d$ and $p, q \in H(l)$ are at least
$\norm{d}$ apart, then the line segment from $p$ to $q$ contains at
least one point of every line through~$l$.

\begin{proof}[Proof of Lemma \ref{injective global dependencies}]
Without loss of generality we restrict ourselves to
$A = \Delta \cap \Z^2$.
We want to show that for any $l \in L(\Z^2, k)$, there is an
$l' \in L(A,k)$ that maps to the same element
in $\cok{\sigma_{\Z^2, k}}$. That is, we must show
\[ L(\Z^2, k) = \im{\sigma_{\Z^2, k}} + L(A,k).\]
Recall that the conductor
\[
\ff_k =
\overline{D_1} \quo{k[\Z^2]}{d_1 - 1}
\oplus\cdots\oplus
\overline{D_t} \quo{k[\Z^2]}{d_t - 1}
\]
is the largest $L(\Z^2, k)$ ideal that is contained in
$\im{\sigma_{\Z^2,k}}$. It is therefore sufficient
to show that $ L(\Z^2, k) = \ff_k + L(A,k)$, or, equivalently,
that
\[\freemod{k}{\L_i(A)} \lra \quo{k[\Z^2]}{(d_i - 1, D_i)}\]
is surjective for all $i$.

Let $l$ be a geometric line in direction $d_i$ such that $H(l)$
intersects $\Delta$. As we have $\Delta = P(D_i) + P(d_i - 1)$, the
intersection is a segment of width at least $\norm{d_i}$, so every line
in the direction $d_i$ that lies in $l$ is in $\L_i(A)$.
Let $S \subset \Z^2$ be the union of all the lines in $\L_i(A)$.

Note that $P(D_i)$ does not have a side parallel to $d_i$, as all the
directions are pairwise independent. It follows that $P(D_i)$ has
maximal points in the directions orthogonal to $d_i$. These points are
nescesarily corners. The coefficients on these corners are $\pm1$.
It follows that for any $f \in k[\Z^2]$, there is a $g \in k[\Z^2]$ such that
$\supp{g} \subset S$ and $f - g \in D_i k[\Z^2]$.

By the above, this implies that
\[\freemod{k}{\L_i(A)} + \overline{D_i} \quo{k[\Z^2]}{d_i - 1}
= \quo{k[\Z^2]}{d_i - 1}\]
and so
\[\freemod{k}{\L_i(A)} \lra \quo{k[\Z^2]}{(d_i - 1, D_i)}\]
is surjective.
\end{proof}

Let $A$ be finite and convex.
We define the \emph{rounded part} of $A$ to be the subset
\[A' = \left(\bigcup (x + \Delta)\right) \cap \Z^2\]
where the union runs over all $x \in \Z^2$ such that $x + \Delta \subset
H(A)$. We call $A$ \emph{rounded} if it is non-empty and $A' = A$.

\begin{theorem}\label{rounded dependencies}
Let $A$ be finite, convex and rounded. Then $\cok{\sigma_{A,k}}$ is
equal to $\cok{\sigma_{\Z^2, k}}$ and so we have
\[\Dep{A,k} = \Dep{\Z^2, k}.\]
\end{theorem}
\begin{proof}
Note that by Lemma \ref{injective global dependencies} the map
\[\cok{\sigma_{A,k}} \lra \cok{\sigma_{\Z^2,k}}\]
is surjective, so we just have to show it is injective.
The strategy for this is to construct
\[A = A_0 \subset A_1 \subset A_2 \subset \cdots\]
such that $A_{i+1} / A_i$ is non-interfering for all $i \ge 0$ and
$\bigcup_{i \ge 0} A_i$ is all of $\Z^2$. Suppose that $l \in L(A, k)$
such that $l = \sigma_{\Z^2, k}(t)$ for some $t \in T(\Z^2, k)$. Then
$t \in T(A_i, k)$ for some $i$, so $l$ maps to $0$ in
$\cok{\sigma_{A_i, k}}$. By the non-interference, $\cok{\sigma_{A,k}}$
maps injectively to $\cok{\sigma_{A_i, k}}$, so it follows that $l$
maps to $0$ in $\cok{\sigma_{A,k}}$, as required.

Pick a point $p$ in the interior of $H(A)$ in sufficiently general
position (we will make this more precise later on). For $\lambda \in
\R_{\ge1}$ let $H(\lambda)$ be the point multiplication of the set
$H(A)$ with factor $\lambda$ and center $p$. Let $A(\lambda) =
H(\lambda) \cap \Z^2$. Note that the union of all $H(\lambda)$ is the
entire plane, so we have
\[\bigcup\limits_{\lambda \ge 1} A(\lambda) = \Z^2.\]

As $\Z^2 \subset \R^2$ is countable and discrete, the set of $\lambda$'s
such that
\[A(\lambda) \neq \bigcup\limits_{1 \le \mu < \lambda} A(\mu)\]
is a countable and discrete subset of $\R_{\ge 1}$. Let
$(\lambda_i)_{i=0}^\infty$ be the sequence of these $\lambda$'s in
increasing order. Put $A_i = A(\lambda_i)$.

For all $\lambda \in \R_{\ge 1}$ one sees that
\[\bigcup\limits_{1 \le \mu < \lambda} H(\mu)\]
is the boundary of $H(\lambda)$. Therefore, any point in $A_{i+1} \setminus
A_{i}$ is on the boundary of $H(\lambda_i)$. This means that these
points lie on finitely many line segments: the edges of the polygon
$H(\lambda_{i+1})$.

In fact, by choosing the point $p$ outside a countable union of lines,
one can ensure that for every $i$ there is a single edge $l_i$ of the
polygon $H(\lambda_{i+1})$ such that all the points in
$A_{i+1} \setminus A_i$ lie on that edge.

Suppose that $l_i$ does not lie in one of the directions $d_1, \ldots
d_t$. Then $\Delta$ has a maximal point $m$ in the direction orthogonal to
$l_i$, which is a corner and so the corresponding coefficient of $D$ is
$\pm1$. Let $p \in A_{i+1} \setminus A_i$. As $A$ is rounded, the
translate of $\Delta$ such that $m$ coincides with $p$ is contained
entirely in $A_{i+1}$. It follows that the map
\[\ker{\sigma_{A_{i+1},k}} \lra \freemod{k}{A_{i+1}\setminus A_i}\]
is surjective, so $A_{i+1}/A_i$ is non-interfering.

Suppose that $l_i$ lies in the direction $d_j$. The edge of $H(A)$ in
direction $d_j$ is at least $\norm{d_j}$ long, as $A$ is rounded.
So the edge $l_i$ of $H(\lambda_{i+1})$ has length
$\lambda_{i+1} |d_j| > |d_j|$. Therefore, every line in the direction
$d_j$ that lies inside the geometric line containing $l_i$ meets
$A_{i+1}$. Note that $\Delta$ has an edge in direction $d_j$ and that
the intersection of $\supp{D}$ with the geometric line through that edge
consists precisely of the two corner points, both of which have
coefficient $\pm1$. These two points are adjacent points within the same
line on that geometric line. As $A$ is rounded, every translate of
$\Delta$ such that the edge in direction $d_j$ lies between on $l_i$,
lies completely within $H(A)$. From these observations we can conclude
that
\[
\sigma_{A_{i+1}/A_i, k} : T(A_{i+1}/A_i, k) \lra L(A_{i+1}/A_i, k)
\]
is onto and that its kernel is generated by the intersections of the correct
translates of $D$ with $T(A_{i+1}/A_i, k)$. Therefore the map
\[
\ker{\sigma_{A_{i+1},k}} \lra \ker{\sigma_{A_{i+1}/A_i,k}}
\]
is onto, that is, $A_{i+1}/A_i$ is non-interfering.
\end{proof}

\begin{theorem}
\label{local_dependencies}
Let $A$ be finite and convex and suppose that $A'$ is non-empty. Then
$\Dep{A,k}$ decomposes in a natural way as a direct sum
\[ \Dep{A, k} = \Dep{\Z^2, k} \oplus \Hom[k]{\cok{\sigma_{A/A',k}}, k}.\]
We call the second summand the \emph{local dependencies} on $A$.
\end{theorem}
\begin{proof}
From the comparison sequence for $A/A'$ we have
\[
\cok{\sigma_{A',k}} \lra_{f_{A/A'}} \cok{\sigma_{A,k}} \lra
\cok{\sigma_{A/A',k}} \lra 0.
\]
Lemma \ref{injective global dependencies} shows
$f_A : \cok{\sigma_{A,k}} \ra \cok{\sigma_{\Z^2,k}}$
is surjective and Theorem \ref{rounded dependencies}
shows $f_{A'} : \cok{\sigma_{A',k}} \ra \cok{\sigma_{\Z^2,k}}$
is bijective. Note that  $f_{A'} = F_{A} \circ f_{A/A'}$.
We conclude that $f_{A/A'}$ is injective (so $A/A'$ is non-interfering)
and that $f_{A'}^{-1} \circ f_{A}$ is a splitting map of $f_{A/A'}$.
It follows that
\[ \cok{\sigma_{A,k}} = \cok{\sigma_{A',k}} \oplus
\cok{\sigma_{A/A',k}}.\]
This implies the required result (recall that $\Dep{A', k} = \Dep{\Z^2,
k}$.)
\end{proof}

\section{Conclusions}%
To conclude this paper, we summarize the main results obtained within
our algebraic framework, and their interpretation from the classical
combinatorial perspective.

Lemma \ref{dependency_cok_lemma} relates an algebraic property of the
cokernel of the line sum map to the consistency problem. Theorem
\ref{global_kernel_and_cokernel} states that for the case $A =
\mathbb{Z}^2$, the cokernel actually satisfies this property.  In
addition, a characterization of the switching components is provided for
this case.  This results in a strong statement concerning the
consistency problem for the case $A = \mathbb{Z}^2$: a set of linesums
corresponds to a table if and only if it satisfies a certain number of
independent dependencies (Corollary \ref{global_consistency_problem}).
In Section \ref{global_ring_section}, properties are derived on the
structure of the coefficients in the separate dependencies. Section
\ref{example_section} relates the material from Section
\ref{algebraic_notation_section}, \ref{global_case_section} and
\ref{global_ring_section} to the example from the Combinatorial DT
literature, given in Section \ref{global_example}.

The next sections, starting with Section
\ref{comparison_sequence_section}, focus on cases where $A$ is a true
subset of $\mathbb{Z}^2$.  A relative setup is introduced in Section
\ref{comparison_sequence_section}, where a DT problem on a particular
domain is related to a problem on a subset of that domain. In Sections
\ref{finite_convex_section} and \ref{local_global_dependencies_section},
this relation is applied to describe the structure of line sums for
finite convex sets. Corollary \ref{global_consistency_problem} provides
a necessary and sufficient condition for consistency in the case of a
finite, convex reconstruction domain. Theorem \ref{rounded dependencies}
shows that if $A$ is finite, convex and rounded, the dependencies are
exactly those that also apply to the global case $A = \mathbb{Z}^2$.
Finally, Theorem \ref{local_dependencies} considers the decomposition of
the dependencies for the general finite convex case into global and
local dependencies.

The results on the structure of dependencies between the line sums in
discrete tomography problems can either be viewed as a collection of new
research results, or as an illustration of the power of applying Ring
Theory and Commutative Algebra to this combinatorial problem. We expect
that a range of additional results can be obtained within the context of
this algebraic framework.

\appendix
\section{Tools from algebra}%
\label{tools_algebra_section}

\subsection{Group rings}
We begin by recalling some results on group rings. See for example
\cite[Ch. II.3]{Lang} for a short introduction or \cite{Passman} for
more results on these rings.

\begin{definition}
Let $k$ be a commutative ring and $G$ be a group. The \emph{group ring}
$k[G]$ is the $k$-algebra which as a $k$-module is the free with basis $G$,
\[
k[G] = \bigoplus_{g \in G} k [g]
\]
and whose multiplication is given by
\[
\begin{array}{ccl}
[g] \cdot [h] = [gh] & \qquad   & \mbox{for all $g, h \in G$} \\ \relax
[g] \cdot \lambda = \lambda[g] && \mbox{for all $g \in G$,
$\lambda \in k$}.
\end{array}
\]
\end{definition}

When there is no confusion possible we will drop the brackets around
elements of $G$, writing a typical element of $k[G]$ simply as
$\sum_{g \in G} \lambda_g g$ with $\lambda_g = 0$ for almost all
$g \in G$.

A ring homomorphism $k \lra k'$ induces a unique ring
homomorphism
\[k[G] \lra k'[G].\]
A group homomorphism $G \lra H$ induces a unique $k$-algebra
homomorphism
\[k[G] \lra k[H].\]

\begin{lemma}\label{groupring quotients}
Let $G$ be a group and $N$ be a normal subgroup. Let $I_N$ be the ideal
of $k[G]$ generated by all elements of the form $n - 1$ with $n \in N$.
Then there is a short exact sequence
\[0 \lra I_N \lra k[G] \lra k[\quo{G}{N}] \lra 0.\]
\end{lemma}

\subsection{Filtrations}
We continue with some generalities on filtrations.

\begin{definition}Let $R$ be a commutative ring and $M$ an $R$-module. A
\emph{filtration} of $M$ is a collection of submodules
\[0 = M_0 \subset M_1 \subset \cdots \subset M_t = M.\]
The quotient modules $M_{i+1}/M_i$ are called the \emph{successive
quotients} of the filtration.
\end{definition}

\begin{lemma} Let $R$ be a commutative ring and let $M'$ and $M''$ be
filtered $R$-modules. Suppose we have a short exact sequence
\[0\ra M'\ra M \ra M''\ra 0.\]
Then $M$ admits a filtration whose successive quotients are those of
$M'$ followed by those of $M''$
\end{lemma}

\begin{lemma}\label{composition of injectives}
Let $R$ be a commutative ring, let $A, B$ and $C$ be $R$-modules and
suppose $f : A \ra B$ and $g: B \ra C$ are injective morphisms. Then
there is a short exact sequence
\[ 0 \ra \cok{f} \ra \cok{gf} \ra \cok{g} \ra 0.\]
\end{lemma}

\subsection{Weak coprimality}
The rest of this appendix is devoted to a generalisation of the concept
of coprimality and the Chinese Remainder Theorem.

\begin{definition}\label{weakly coprime}
Let $R$ be a commutative ring and let $f, g \in R$. We say that $f$ is
\emph{weakly coprime to $g$} if multiplication by $f$ is an injective
map on $\quo{R}{g}$.
\end{definition}

The common notion of coprimality, namely that the ideal $(f, g)$
generated by $f$ and $g$ be all of $R$, implies that multiplication by
$f$ is a \emph{bijective} map on $\quo{R}{g}$.

\begin{lemma}\label{wcrt2}%
Let $R$ be a commutative ring and let $f, g \in R$ such that $f$ is
weakly coprime to $g$.  Then there is a short exact sequence
\[0 \ra \quo{R}{fg} \ra \quo{R}{f} \oplus \quo{R}{g} \ra
\quo{R}{(f, g)} \ra 0.\]
\end{lemma}
\begin{proof} Straightforward verification.\end{proof}

If two elements are coprime in the common (strong) sense, then in the
sequence above we have $R/(f,g) = 0$, so the first map is an
isomorphism.  This fact is commonly refered to as the Chinese Remainder
Theorem.

\begin{lemma}\label{product filtration}
Let $R$ be a commutative ring and let $f_1, f_2$ and $g$ be in $R$.
Suppose that $f_1$ and $f_2$ are weakly coprime to $g$. Then there is a
short exact sequence
\[0 \ra \quo{R}{(f_1, g)} \ra \quo{R}{(f_1f_2, g)} \ra \quo{R}{(f_2,g)}
\ra 0.\]
\end{lemma}
\begin{proof}
Apply lemma \ref{composition of injectives} to the multiplication by
$f_1$ and by $f_2$ maps on $R/g$.
\end{proof}

\begin{theorem}[Weak Chinese Remainder Theorem]\label{weak chinese
remainder theorem}%
Let $R$ be a commutative ring and let $x_1, \cdots, x_t \in R$ have the
property that $x_i$ is weakly coprime to $x_j$ whenever $i < j$. Then
the natural map
\[ \phi: \quo{R}{x_1 \cdots x_t}\longrightarrow
\quo{R}{x_1} \oplus \cdots \oplus \quo{R}{x_t}
\]
is injective. Its cokernel admits a filtration whose successive
quotients are isomorphic to $\quo{R}{(x_i, x_j)}$ for $1 \le i < j \le t$.
\end{theorem}
\begin{proof}
We proceed by induction on $t$. For $t = 2$ the result is that of Lemma
\ref{wcrt2}. Let $t \ge 3$ and assume that the theorem holds for any
smaller number of $x_i$'s.

We write $\phi$ as a composition of two maps. Let $\phi_1$ be the
natural map
\[\phi_1 :
\quo{R}{x_1 \cdots x_t} \lra
\quo{R}{x_1 \cdots x_{t-1}} \oplus \quo{R}{x_t}\]
and $\phi_2$ the natural map
\[\phi_2 :
\quo{R}{x_1 \cdots x_{t-1}} \lra
\quo{R}{x_1} \oplus \cdots \oplus \quo{R}{x_{t-1}}.\]
Then we have $\phi = (\phi_2 \oplus \id_{R/x_t}) \circ \phi_1$.

Note $x_1 \cdots x_{t-1}$ is weakly coprime to $x_t$ as a composition of
injective maps is again injective. So Lemma \ref{wcrt2} applies to
$\phi_1$. In particular $\phi_1$ is injective. By the induction
hypothesis, $\phi_2$ is also injective. We conclude that $\phi$ is
injective.

By Lemma \ref{wcrt2} the cokernel of $\phi_1$
is $R/(x_1\cdots x_{t-1}, x_t)$. By repeatedly applying
Lemma \ref{product filtration}, this module admits a filtration whose
successive quotients are $R/(x_i, x_t)$ for $1 \le i \le t-1$.

Furthermore, we have $\cok{\phi_2 \oplus
\id_{R/x_t}} = \cok{\phi_2}$, which by the induction hypothesis has a
filtration whose successive quotients are isomorphic to
$R/(x_i, x_j)$ with $1 \le i < j \le t-1$.

We apply Lemma \ref{composition of injectives} to the maps $\phi_1$
and $\phi_2 \oplus \id_{R/x_t}$ and conclude that the cokernel of $\phi$
has the required filtration.
\end{proof}

\bibliographystyle{plain}
\bibliography{paper_d}

\end{document}